\documentclass[reqno]{amsart}
\usepackage{graphicx}

\usepackage{epstopdf}
\usepackage[justification=centering]{caption}
\usepackage{lineno,hyperref}
\usepackage{amsfonts,amsmath,latexsym,amssymb,amsthm}
\usepackage{geometry}

\newcommand{\hs}{{\mathcal H\,}}

\newcommand{\ns}{\mathbb N\,}
\modulolinenumbers[5]
\usepackage{makecell,multirow,diagbox}
\hypersetup{
	colorlinks=true,
	linkcolor=blue
}
\hypersetup{
	citecolor=blue
}
\newtheorem{definition}{Definition}
\newtheorem{theorem}{Theorem}[section]

\usepackage[justification=centering]{caption}
\newtheorem{corollary}[theorem]{Corollary}
\newtheorem{proposition}{Proposition}
\newtheorem{example}[theorem]{Example}
\theoremstyle{remark}
\newtheorem{remark}[theorem]{Remark}

\numberwithin{equation}{section}
\allowdisplaybreaks[4]

\begin{document}

\title{On weaving frames in Hilbert spaces}

\author{Dongwei Li}
\address{School of Mathematics, HeFei University of Technology, 230009, P. R. China}
\email{dongweili@huft.edu.cn}





\begin{abstract}
	In this paper, we obtain  some new properties of weaving frames and present some conditions under which a family of frames is woven in Hilbert spaces.   Some characterizations of weaving frames in terms of  operators are given. We also give a condition associated with synthesis operators of frames such that the sequence of frames is woven. Finally, for a family of woven frames, we show that  they are stable under invertible operators and small perturbations. 
\end{abstract}

\maketitle

\section{Introduction}
Frames in Hilbert spaces were first introduced by Duffin and Schaeffer \cite{duffin1952class} for studying some problems in nonharmonic Fourier series, reintroduced in 1986 by Daubechies, Grossman, and Meyer \cite{daubechies1986painless} and popularized from then on. Redundancy of frames is one of the key features that are important in both theory and application, and it provides flexibility on constructions of various classes of frames. Just as the nice properties of frames, frames have been applied to wide range of science and technology fields such as signal processing \cite{han2014reconstruction}, coding theory    \cite{casazza2003equal,han2018recovery,li2018frame},  sampling theory \cite{lu2008theory}, quantum measurements \cite{eldar2002optimal} and image processing \cite{chan2004tight}, etc.

Let $\hs$ be a separable space and $I,~J$ a countable index set. A sequence  $\{f_j\}_{i\in J}$ of elements of $\hs$ is a frame for $\hs$ if there exist constants $A,~B>0$ such that
$$A\|f\|^2\le \sum_{j\in J}\left| \left\langle f,f_j\right\rangle \right| ^2\le B\|f\|^2,~~~\forall f\in\hs.$$
The number $A,~B$ are called lower and upper frame bounds, respectively. If $A=B$, then this frame is called an $A$-tight frame, and if $A=B=1$, then it is called a Parseval frame.

Suppose  $\{f_j\}_{j\in J}$ is a frame for $\hs$, then the frame operator is a self-adjoint positive invertible operators, which is given by
$$S:\hs\rightarrow\hs,~~Sf=\sum_{j\in J}\left\langle f,f_j\right\rangle f_j.$$ 
The following reconstruction formula holds:
$$f=\sum_{j\in J}\left\langle f,f_j\right\rangle S^{-1}f_j=\sum_{j\in J}\left\langle f,S^{-1}f_j\right\rangle f_j,$$
where the family $\{\widetilde{f}_j\}_{j\in J}=\{S^{-1}f_j\}_{j\in J}$ is also a frame for $\hs$, which is called the canonical dual frame of $\{f_j\}_{j\in J}$.
The frame $\{g_j\}_{j\in J}$ for $\hs$ is called an alternate dual frame of $\{f_j\}_{j\in J}$ if the following formula holds:
$$f=\sum_{j\in J}\left\langle f,f_j\right\rangle g_j=\sum_{j\in J}\left\langle f,g_j\right\rangle f_j$$
for all $f\in\hs$ \cite{han2000frames}.

Weaving frames were introduced in \cite{bemrose2015weaving} and investigated in \cite{casazza2016weaving,casazza2015weaving}. The concept of weaving frames is motivated by distributed signal processing, which have potential applications in wireless sensor networks that require distributed processing under different frames, as well as pre-processing of signals using Gabor frames. For example, in wireless sensor  network, let  two frames $F=\{f_j\}_{j\in J}$ and $G=\{g_j\}_{j\in J}$ be measures tools. At each sensor, we encode signal $f$ either with $f_j$ or $g_j$, so the encode coefficients is the set of numbers $\{\left\langle f,f_j\right\rangle \}_{j\in\sigma}\cup\{\left\langle f,g_j\right\rangle \}_{j\in\sigma^c}$ for some $\sigma\subset J$.
If $\{f_j\}_{j\in\sigma}\cup\{g_j\}_{j\in\sigma^c}$ is still a frame, then $f$ can be recovered robustly from these coefficients. We say that $F$ and $G$ are woven frames.

 But $\{f_j\}_{j\in\sigma}\cup\{g_j\}_{j\in\sigma^c}$ may be not a frame for $\hs$ for any $\sigma\subset J$. For example, let $\{e_j\}_{j=1}^3$ be an orthonormal basis for $\hs$, $F=\{f_j\}_{j=1}^3=\{e_1, e_2, e_1+e_3\}$ and $G=\{g_j\}_{j=1}^3=\{e_1,e_3,e_1+e_2\}$, then $F$ and $G$ are two frames for $\hs$. If we choose $\sigma=\{1,2\}$, then $\{f_j\}_{j\in\sigma}\cup\{g_j\}_{j\in\sigma^c}=\{e_1, e_2, e_1+e_2\}$ is not a frame for $\hs$.
  
What is the condition such that $\{f_j\}_{j\in\sigma}\cup\{g_j\}_{j\in\sigma^c}$ is a frame for $\hs$ for any $\sigma\subset J$? In this paper, we give some sufficient conditions under that a family of frames is woven in Hilbert spaces, we also consider that perturbation applied to woven frames leaves them woven. 

We first recall some concept and properties of woven frames.
 
\begin{definition}\cite{bemrose2015weaving}
A family of frames $\{F_i=\{f_{ij}\}_{j\in J}\}_{i\in I}$ for $\hs$ is said to be woven if there are universal constants $A$ and $B$ such that for every partition $\{\sigma_i\}_{i\in I}$ of $J$ the family $\{f_{ij}\}_{j\in \sigma_i,i\in I}$ is a frame for $\hs$ with lower and upper frame bounds $A$ and $B$, respectively, where $I=\{1,2,\cdots,m\}$. And $\{f_{ij}\}_{j\in \sigma_i,i\in I}$ called a weaving frame (or a weaving).
\end{definition}
The following proposition gives that every weaving automatically has a universal upper frame bound.
\begin{proposition}\cite{bemrose2015weaving}\label{prop1}
If each $F_i=\{f_{ij}\}_{j\in J}$ is a Bessel sequence for $\hs$ with bounds $B_i$ for all $i\in I$, then every weaving is a Bessel sequence with $\sum_{i\in I}B_i$ as a Bessel bound.
\end{proposition}

Let $\{\sigma_i\}_{i\in I}$ be any partition of $J$, now we define the space:
$$\bigoplus_{i\in I}\ell^2(\sigma_i)=\left\lbrace  \{c_{ij}\}_{j\in\sigma_i,i\in I}|c_{ij}\in \mathbb{C},~\sigma_i\subset J,~i\in I,~\sum_{j\in \sigma_i,i\in I}|c_{ij}|<\infty\right\rbrace ,$$
with the inner product
$$\left\langle \{c_{ij}\}_{j\in\sigma_i,i\in I},\{d_{ij}\}_{j\in\sigma_i,i\in I}\right\rangle =\sum_{j\in \sigma_i,i\in I}|c_{ij}\overline{d_{ij}}|,$$
it is clean that  $\bigoplus_{i\in I}\ell^2(\sigma_i)$ is a Hilbert space.

Let the family of frames $\{F_i=\{f_{ij}\}_{j\in J}\}_{i\in I}$ be woven for $\hs$, for  any partition $\{\sigma_i\}_{i\in I}$ of $J$, $W=\{f_{ij}\}_{j\in\sigma_i,i\in I}$ is a frame for $\hs$, the operator $T_W:\bigoplus_{i\in I}\ell^2(\sigma_i)\rightarrow\hs$ defined by
$$T_W(\{c_{ij}\})=\sum_{i\in I}T_{F_i}D_{\sigma_i}(\{c_{ij}\})=\sum_{i\in I}\sum_{j\in \sigma_i}c_{ij}f_{ij},$$
is called the synthesis operator, where $T_{F_i}$ is the synthesis operator of $F_i$ and $D_{\sigma_i}$ is a $|J|\times|J|$ diagonal matrix with $d_{jj}=1$ for $j\in\sigma_i$ and otherwise 0. The adjoint operator of $T_W$ is given by:
 $$T_W^*:\hs \rightarrow\bigoplus_{i\in I}\ell^2(\sigma_i),~~T_W^*(f)=\sum_{i\in I}D_{\sigma_i}{T_{F_i}^{\sigma_i}}^*(f)=\{\left\langle f,f_{ij}\right\rangle \}_{j\in\sigma_i,i\in I},~~~,\forall f\in\hs,$$
 and is called the analysis operator. The  frame operator $S_W$ is defined as
 $$S_W:\hs\rightarrow\hs, S_W(f)=T_WT_W^*(f)=(\sum_{i\in I}T_{F_i}D_{\sigma_i})(\sum_{i\in I}T_{F_i}D_{\sigma_i})^*(f)=\sum_{i\in I}S_{F_i}^{\sigma_i}(f)=\sum_{j\in \sigma_i,i\in I}\left\langle f,f_{ij}\right\rangle f_{ij},$$
 where $S_{F_i}$ is the frame operator of $F_i$ and $S_{F_i}^{\sigma_i}$ is a ``truncated form'' of $S_{F_i}$.
 The operator $S_W$ is positive, self-adjoint and invertible.
 \section{Main Results}
 We first give some properties of weaving frames. 
 
 \begin{proposition}
Let two frames $F=\{f_j\}_{j\in J}$ and $G=\{g_j\}_{j\in J}$ be woven with synthesis operators $T_F$ and $T_G$, respectively. For any $\sigma\subset J$, a weaving $\{f_j\}_{j\in\sigma}\cup\{g_j\}_{j\in\sigma^c}$ is a $A$-tight frames for $\hs$ if and only if $T_FD_{\sigma}T_F^*+T_GD_{\sigma^c}T_G^*=AI_{\hs}$, which $D_{\sigma}$ is a $|J|\times|J|$ diagonal matrix with $d_{jj}=1$ for $j\in\sigma$ and otherwise 0, $D_{\sigma^c}$ is a $|J|\times|J|$ diagonal matrix with $d_{jj}=1$ for $j\in\sigma^c$ and otherwise 0.
 \end{proposition}
\begin{proof}
For any $\sigma\subset J$, then the synthesis of weaving frame $W=\{f_j\}_{j\in\sigma}\cup\{g_j\}_{j\in\sigma^c}$ is $T_FD_{\sigma}+T_GD_{\sigma^c}$. Then the frame operator 
\begin{align*}
S_W
&=(T_FD_{\sigma}+T_GD_{\sigma^c})(T_FD_{\sigma}+T_GD_{\sigma^c})^*\\
&=T_FD_{\sigma}T_F^*+T_FD_{\sigma}D_{\sigma^c}T_G^*+T_GD_{\sigma^c}D_{\sigma}T_F^*+T_GD_{\sigma^c}T_G^*\\
&=T_FD_{\sigma}T_F^*+T_GD_{\sigma^c}T_G^*.
\end{align*}

\end{proof}
\begin{proposition}\label{pro3}
Let two frames $F=\{f_j\}_{j\in J}$ and $G=\{g_j\}_{j\in J}$ be woven with universal constants $A$ and $B$ frame operators $S_F$ and $S_G$, respectively. If $\|S^{-1}_F\|\|S_F-S_G\|<\frac{A}{B}$ $($or $\|S^{-1}_G\|\|S_F-S_G\|<\frac{A}{B}$$)$, then $S^{-1}_{F}F=\{S^{-1}_Ff_j\}_{j\in J}$ and $S^{-1}_GG=\{S^{-1}_Gg_j\}_{j\in J}$ are also woven.
\end{proposition}
\begin{proof}
We only consider the case of $\|S^{-1}_F\|\|S_F-S_G\|<\frac{A}{B}$. Now for every $\sigma\subset J$ and each $f\in\hs$, we have
\begin{align*}
&\big(\sum_{j\in \sigma}\left| \left\langle f,S^{-1}_Ff_i\right\rangle \right|^2 +\sum_{j\in \sigma^c}\left| \left\langle f,S^{-1}_Gg_i\right\rangle \right|^2\big)^{1/2}\\
&=\big(\sum_{j\in \sigma}\left| \left\langle S^{-1}_Ff,f_i\right\rangle \right|^2 +\sum_{j\in \sigma^c}\left| \left\langle S^{-1}_Gf,g_i\right\rangle \right|^2\big)^{1/2}\\
&=\big(\sum_{j\in \sigma}\left| \left\langle S^{-1}_Ff,f_i\right\rangle \right|^2 +\sum_{j\in \sigma^c}\left| \left\langle S^{-1}_Ff+(S^{-1}_G-S^{-1}_F)f,g_i\right\rangle \right|^2\big)^{1/2}\\
&\ge \big(\sum_{j\in \sigma}\left| \left\langle S^{-1}_Ff,f_i\right\rangle \right|^2 +\sum_{j\in \sigma^c}\left| \left\langle S^{-1}_Ff,g_i\right\rangle \right|^2\big)^{1/2}-\big(\sum_{j\in \sigma^c}\left| \left\langle (S^{-1}_G-S^{-1}_F)f,g_i\right\rangle \right|^2\big)^{1/2}\\
&\ge \sqrt{A}\|S^{-1}_Ff\|-\big(\sum_{j\in J}\left| \left\langle (S^{-1}_G-S^{-1}_F)f,g_i\right\rangle \right|^2\big)^{1/2}\\
&\ge \sqrt{A}\|S^{-1}_Ff\|-\sqrt{B}\|S^{-1}_G-S^{-1}_F\|\|f\|\\
&\bigg(\frac{\sqrt{A}}{\|S^{-1}_F\|}-\sqrt{B}\|S^{-1}_G-S^{-1}_F\|\bigg)\|f\|,
\end{align*}
and 
\begin{align*}
\sum_{j\in \sigma}\left| \left\langle f,S^{-1}_Ff_i\right\rangle \right|^2 +\sum_{j\in \sigma^c}\left| \left\langle f,S^{-1}_Gg_i\right\rangle \right|^2&\le \sum_{j\in J}\left| \left\langle f,S^{-1}_Ff_i\right\rangle \right|^2 +\sum_{j\in J}\left| \left\langle f,S^{-1}_Gg_i\right\rangle \right|^2\\
&\le B(\|S_F\|^2+\|S_G\|^2)\|f\|^2.
\end{align*}
\end{proof}  

In Proposition 3, for any $\sigma\subset J$, $\{f_j\}_{j\in\sigma}\cup\{g_j\}_{j\in\sigma^c}$ is a frame for $\hs$, and $\{S^{-1}_Ff_j\}_{j\in\sigma}\cup\{S^{-1}_Gg_j\}_{j\in\sigma^c}$
is also a frame, it should be noted that $\{S^{-1}_Ff_j\}_{j\in\sigma}\cup\{S^{-1}_Gg_j\}_{j\in\sigma^c}$ is not a dual frame of $\{S^{-1}_Ff_j\}_{j\in\sigma}\cup\{S^{-1}_Gg_j\}_{j\in\sigma^c}$ in general.
\begin{example}
For two given frames  $F=\{f_j\}_{j=1}^3$, $G=\{g_j\}_{j=1}^3$,
$$
F=\left\{ {\left[ {\begin{array}{*{20}c}
		{1 }  \\
		0 \\
		\end{array}} \right],\left[ {\begin{array}{*{20}c}
		0  \\
		1  \\
		\end{array}} \right],\left[ {\begin{array}{*{20}c}
		1 \\
		1  \\
		\end{array}} \right]} \right\},~~~G=\left\{ {\left[ {\begin{array}{*{20}c}
	1  \\
		0 \\
		\end{array}} \right],\left[ {\begin{array}{*{20}c}
		1  \\
		1\\
		\end{array}} \right],\left[ {\begin{array}{*{20}c}
		1 \\
		-1  \\
		\end{array}} \right]} \right\}.
$$ 
It is easy to verify that $F$ and $G$ are woven, and $S_F^{-1}F$ and $S_G^{-1}G$ are woven. Suppose $\sigma=\{1,2\}$, then the weaving is  given by
$$W=\left\{ {\left[ {\begin{array}{*{20}c}
			{1 }  \\
			0 \\
	\end{array}} \right],\left[ {\begin{array}{*{20}c}
			0  \\
			1  \\
	\end{array}} \right],\left[ {\begin{array}{*{20}c}
			1 \\
			-1  \\
	\end{array}} \right]} \right\},~~~\widetilde{W}=\{S_F^{-1}f_1,S_F^{-1}f_2,S_G^{-1}g_3\}=\left\{ {\left[ {\begin{array}{*{20}c}
	\frac{2}{3}  \\
	-\frac{1}{3}  \\
	\end{array}} \right],\left[ {\begin{array}{*{20}c}
	-\frac{1}{3}  \\
	\frac{2}{3}  \\
	\end{array}} \right],\left[ {\begin{array}{*{20}c}
	\frac{1}{3} \\
	-\frac{1}{2}  \\
	\end{array}} \right]} \right\}.$$
We compute $T_WT_{\widetilde{W}}^*=\left[ {\begin{array}{*{20}c}
	{1 }  &-\frac{5}{6}\\
	-\frac{2}{3} & \frac{7}{6}\\
	\end{array}} \right]\ne I_{2\times2}$, thus $\widetilde{W}$ is not a dual frame of $W$.
\end{example}
The following result give a simple characterization of dual frames of a weaving. 
\begin{proposition}\label{prop4}
Suppose that the family of frames $\{F_i=\{f_{ij}\}_{j\in J}\}_{i\in I}$ is woven.  Let $\{\sigma_i\}_{i\in I}$ be any partition of $J$, then the dual frames of weaving frames $W=\{f_{ij}\}_{j\in\sigma_i,i\in I}$ is given by $S_W^{-1}T_W+U$, where $\sum_{i\in I}T_{F_i}D_{\sigma_i}U^*=0$.
\end{proposition}
\begin{proof}
A simple calculation yields this.
\end{proof}
\begin{remark}
In Proposition \ref{prop4}, we call $S_W^{-1}T_W+U$ the alternative dual of $W$; if $U=0$, then $S_W^{-1}$ is called canonical dual of $W$. Thus, the dual of a weaving is similar to the dual of traditional frames.
\end{remark}
We first give a characterization of weaving frames in terms of an operator.
\begin{theorem}
For $i\in I$, let $F_i=\{f_{ij}\}_{j\in J}$ be a sequence for $\hs$. The following conditions are equivalent:
\begin{enumerate}
	\item[(i)] The family of sequences $\{F_i\}_{i\in I}$ is woven frames for $\hs$.
	\item[(ii)]  There exists $A>0$ such that there exists a bounded linear operator $T:\bigoplus_{i\in I}\ell^2(\sigma_i)\rightarrow\hs$ such that $T(u_{ij})=f_{ij}$ for all $j\in\sigma_i,i\in I$, and $AI_{\hs}\le TT^*$, where $\{u_{ij}\}_{j\in\sigma_i,i\in I}$ is the standard orthonormal basis for $\bigoplus_{i\in I}\ell^2(\sigma_i)$.
\end{enumerate}
\end{theorem}
\begin{proof}
	
  (i)$\Rightarrow$(ii): Suppose $A$ is a universal lower frame bound for the family of sequences $\{F_i\}_{i\in I}$, let $T_W$ be the synthesis operator  associated with $W=\{f_{ij}\}_{j\in\sigma_i,i\in I}$.
  
  Choose $T=T_W$, then 
 $$T(u_{ij})=T_{W}(u_{ij})=\sum_{i\in I}T_{F_i}D_{\sigma_i}(u_{ij})=f_{ij},~~~j\in\sigma_i,i\in I,$$
where $\{u_{ij}\}_{j\in\sigma_i,i\in I}$ is the standard orthonormal basis for $\bigoplus_{i\in I}\ell^2(\sigma_i)$. 

Furthermore, for all $f\in\hs$, we have
$$A\left\langle f,f\right\rangle =A\|f\|^2\le \sum_{j\in \sigma_i,i\in I}\left| \left\langle f,f_{ij}\right\rangle \right| ^2=\|T_W^*(f)\|^2=\|T^*(f)\|^2=\left\langle TT^*f,f\right\rangle .$$
This gives $AI_{\hs}\le TT^*$.

(ii)$\Rightarrow$(i): For  any partition $\{\sigma_i\}_{i\in I}$ of $J$, for $\{c_{ij}\}\in \bigoplus_{i\in I}\ell^2(\sigma_i)$ and $T:\bigoplus_{i\in I}\ell^2(\sigma_i)\rightarrow\hs$, we have
\begin{align*}
\left\langle T(\{c_{ij}\}),f\right\rangle =\left\langle T(\sum_{j\in \sigma_i,i\in I}{c_{ij}u_{ij}}),f\right\rangle =\left\langle \sum_{j\in \sigma_i,i\in I}{c_{ij}Tu_{ij}},f\right\rangle=\left\langle \sum_{j\in \sigma_i,i\in I}{c_{ij}f_{ij}},f\right\rangle=\sum_{j\in \sigma_i,i\in I}c_{ij}\left\langle f,f_{ij}\right\rangle .
\end{align*}
This gives 
\begin{equation}\label{2.2}
T^*(f)=\left\langle f,f_{ij}\right\rangle_{j\in\sigma_i,i\in I},~~\forall f\in\hs.
\end{equation}
Since $AI_{\hs}\le TT^*$, by using \eqref{2.2}, we have
$$A\|f\|^2\left\langle TT^*f,f\right\rangle =\|T^*(f)\|^2=\sum_{i\in I}\sum_{j\in \sigma_i}\left| \left\langle f,f_{ij}\right\rangle \right| ^2.$$
On the other hand, for any $f\in\hs$, let $\sigma_i=J$,
we have
$$\sum_{j\in \sigma_i}\left| \left\langle f,f_{ij}\right\rangle \right| ^2=\|T^*f\|^2\le \|T^*\|^2\|f\|^2,$$
this shows a upper bound of $F_i$. Then, by applying Proposition \ref{prop1} we can obtain a universal upper frame bound. Hence, $\{f_{ij}\}_{j\in\sigma_i,i\in I}$ is a frame for $\hs$ and the family of sequences $\{F_i\}_{i\in I}$ is woven.
\end{proof}
\begin{corollary}
The family of sequences $\{F_i=\{f_{ij}\}_{j\in J}\}_{i\in I}$ is woven if and only if for any partition $\{\sigma_i\}_{i\in I}$ of $J$, the synthesis operator of $\{f_{ij}\}_{j\in\sigma_i,i\in I}$ is a well-deﬁned and bounded mapping.
\end{corollary}
The following result shows a sufficient condition such that two Bessel sequence are woven.
\begin{theorem}\label{thm2.3}
Let $F=\{f_j\}_{j\in J}$ and $G=\{g_j\}_{j\in J}$ be two Bessel sequence for $\hs$ with synthesis operator $T_F$ and $T_G$. If for any $\sigma\subset J$,  $I_{\hs}=T_FT_G^*=T_GT_F^*$ and $T_F^{\sigma}{T_G^{\sigma}}^*=T^{\sigma}_G{T_F^{\sigma}}^*$, then $F$ and $G$ are woven frames for $\hs$, which $T_F^{\sigma}$ and $T_G^{\sigma}$ are ``truncated form'' of $T_F$ and $T_G$ for $\sigma\subset J$, respectively.
\end{theorem}
\begin{proof}
Let $B_1$ and $B_2$ be Bessel bounds for $F$ and $G$, respectively. For any $f\in\hs$, $\sigma\subset J$, we have $f=\sum_{j\in J}\left\langle f,f_j\right\rangle g_j=\sum_{j\in J}\left\langle f,g_j\right\rangle f_j$, and $\sum_{j\in \sigma}\left\langle f,f_j\right\rangle g_j=\sum_{j\in \sigma}\left\langle f,g_j\right\rangle f_j$. By using $(a+b)^2\le 2a^2+2b^2$, we compute
\begin{align*}
\|f\|^4&=\left| \left\langle f,f\right\rangle \right| ^2\\
&=\bigg| \bigg\langle \sum_{j\in J}\left\langle f,f_j\right\rangle g_j,f\bigg\rangle \bigg| ^2\\
&=\bigg| \bigg\langle \sum_{j\in \sigma}\left\langle f,f_j\right\rangle g_j+\sum_{j\in \sigma^c}\left\langle f,f_j\right\rangle g_j,f\bigg\rangle \bigg| ^2\\
&\le 2\bigg| \bigg\langle \sum_{j\in \sigma}\left\langle f,f_j\right\rangle g_j,f\bigg\rangle \bigg| ^2+2\bigg| \bigg\langle \sum_{j\in \sigma^c}\left\langle f,f_i\right\rangle g_j,f\bigg\rangle \bigg| ^2\\
&=2\bigg| \bigg\langle \sum_{j\in \sigma}\left\langle f,f_j\right\rangle g_j,f\bigg\rangle\bigg| ^2+2\bigg| \bigg\langle \sum_{j\in \sigma^c}\left\langle f,g_i\right\rangle f_j,f\bigg\rangle \bigg| ^2\\
&=2\bigg| \sum_{j\in \sigma}\left\langle f,f_j\right\rangle \langle g_j,f\rangle\bigg|^2+2\bigg|\sum_{j\in \sigma^c}\left\langle f,g_j\right\rangle \langle  f_j,f\rangle \bigg| ^2\\
&\le 2\sum_{j\in \sigma}\left| \left\langle f,f_j\right\rangle \right| ^2\sum_{j\in \sigma}\left| \left\langle f,g_j\right\rangle \right| ^2+2\sum_{j\in \sigma^c}\left| \left\langle f,g_j\right\rangle \right| ^2\sum_{j\in \sigma^c}\left| \left\langle f,f_j\right\rangle \right| ^2\\
&\le 2B_2\|f\|^2\sum_{j\in \sigma}\left| \left\langle f,f_j\right\rangle \right| ^2+2B_1\|f\|^2\sum_{j\in \sigma^c}\left| \left\langle f,g_j\right\rangle \right| ^2\\
&\le 2\max\{B_1,B_2\}\|f\|^2\bigg( \sum_{j\in \sigma}\left| \left\langle f,f_j\right\rangle \right| ^2+\sum_{j\in \sigma^c}\left| \left\langle f,g_j\right\rangle \right| ^2\bigg).
\end{align*}
Therefore, for all $f\in\hs$, we have
$$\frac{1}{2\max\{B_1,B_2\}}\|f\|^2\le \sum_{j\in \sigma}\left| \left\langle f,f_j\right\rangle \right| ^2+\sum_{j\in \sigma^c}\left| \left\langle f,g_j\right\rangle \right| ^2\le (B_1+B_2)\|f\|^2.$$
Hence, $F$ and $G$ are woven.
\end{proof}
\begin{example}
Let $\{e_j\}_{j\in \ns}$ be an orthonormal basis for $\hs$, and let $f_j=\frac{1}{2}e_j$ and $g_j=2e_j$ for all $j\in J$. For any $f\in\hs$, we have
$$f=\sum_{j\in ns}\left\langle f,e_j\right\rangle e_j=\sum_{j\in J}\left\langle f,\frac{1}{2}e_j\right\rangle 2e_j=\sum_{j\in J}\left\langle f,2e_j\right\rangle \frac{1}{2}e_j=\sum_{j\in J}\left\langle f,f_j\right\rangle g_j=\sum_{j\in J}\left\langle f,g_j\right\rangle f_j,$$
and $\left\langle f,f_j\right\rangle g_j=\left\langle f,g_j\right\rangle f_j$, and then $\sum_{j\in \sigma}\left\langle f,f_j\right\rangle g_j=\sum_{j\in \sigma}\left\langle f,g_j\right\rangle f_j$.
Hence by  Theorem \ref{thm2.3}, $F$ and $G$ are woven. 
In fact, for any $\sigma\subset \ns$, 
$$\sum_{j\in \sigma}\left| \left\langle f,f_j\right\rangle \right| ^2+\sum_{j\in \sigma^c}\left| \left\langle f,g_j\right\rangle \right| ^2=\sum_{j\in \ns}\left| \left\langle f,d_j^{\sigma}f_j+(1-d_j^{\sigma})g_j\right\rangle \right| ^2=\sum_{j\in \ns}\left| \left\langle f,\frac{d_i^{\sigma}}{2}e_j+2(1-d_j^{\sigma})e_j\right\rangle \right| ^2,$$
where $d_j^{\sigma}=1$ for $j\in\sigma$ and otherwise $0$.
Thus $\{f_j\}_{j\in\sigma}\cup\{g_j\}_{j\in\sigma^c}$ is a frame for $\hs$ with bounds $\frac{1}{4}$ and $4$.
\end{example}
\begin{theorem}
Let $F=\{f_j\}_{j\in J}$ be a frame for $\hs$ with bounds $A,~B$, and $\{U_i\}_{i\in I}\subset L(\hs)$. For any $k\in I$, if $U_k$ has a left inverse $V\in L(\hs)$ and $\max_{i\ne k}\|U_k-U_i\|<\sqrt{\frac{A}{(m-1)B}}\cdot\frac{1}{\|V\|}$, then the family of frames $\{U_iF\}_{i\in I}$ is woven.
\end{theorem} 
\begin{proof}
Since $U_i\in L(B)$, we know that $\{U_if_j\}_{j\in J}$ is also a frame for $\hs$. In fact, for $k\in I$, since $VU_k=I_{\hs}$, then $\|I_{\hs}-VU_i\|=\|V(U_k-U_i)\|<\sqrt{\frac{A}{B}}\le 1$. Therefore, $VU_i$ is invertible for $i\in I$. Consequently, $U_i$ has a left inverse. Hence, $\{U_if_j\}_{j\in J}$ is a frame for any $i\in I$. Let $\{\sigma_i\}_{i\in I}$ be any partition of $J$. Then, for every $f\in\hs$ we have 
\begin{align*}
\sum_{i\in I}\sum_{j\in \sigma_i}\left|\left\langle f,U_if_j\right\rangle \right|^2&=\sum_{i\in I}\sum_{j\in \sigma_i}\left|\left\langle U_i^*f,f_j\right\rangle \right|^2\le \sum_{i\in I}\sum_{j\in J}\left|\left\langle f,U_if_j\right\rangle \right|^2\\
&\le B\sum_{i\in I}\|U_i\|^2\|f\|^2.
\end{align*}
On the other hand,
\begin{align*}
&\sum_{i\in I}\sum_{j\in \sigma_i}\left|\left\langle f,U_if_j\right\rangle \right|^2\\
&=\sum_{j\in \sigma_1}\left|\left\langle f,U_1f_j\right\rangle \right|^2+\cdots+\sum_{j\in \sigma_i}\left|\left\langle f,U_if_j\right\rangle \right|^2+\cdots+\sum_{j\in \sigma_k}\left|\left\langle f,U_kf_j\right\rangle \right|^2+\cdots+\sum_{j\in \sigma_m}\left|\left\langle f,U_mf_j\right\rangle \right|^2\\
&=\sum_{j\in \sigma_1}\left|\left\langle f,U_1f_j\right\rangle \right|^2+\cdots+\sum_{j\in \sigma_i}\left|\left\langle f,(U_i-U_k+U_k)f_j\right\rangle \right|^2+\cdots+\sum_{j\in \sigma_k}\left|\left\langle f,U_kf_j\right\rangle \right|^2+\cdots+\sum_{j\in \sigma_m}\left|\left\langle f,U_mf_j\right\rangle \right|^2\\
&\ge \sum_{j\in \sigma_k}\left|\left\langle f,U_kf_j\right\rangle \right|^2+\sum_{j\in\sigma_i}\left|\left\langle f,U_kf_j\right\rangle \right|^2-\sum_{j\in \sigma_i}\left|\left\langle f,(U_i-U_k)f_j\right\rangle \right|^2+\sum_{j\in \sigma_1}\left|\left\langle f,U_1f_j\right\rangle \right|^2+\cdots+\sum_{j\in \sigma_m}\left|\left\langle f,U_mf_j\right\rangle \right|^2\\
&\ge \sum_{j\in J}\left|\left\langle f,U_kf_j\right\rangle \right|^2-\sum_{i\ne k}\sum_{j\in \sigma_i}\left|\left\langle f,(U_i-U_k)f_j\right\rangle \right|^2\\
&\ge A\|U_kf\|^2-\sum_{i\ne k}\sum_{j\in J}\left|\left\langle f,(U_i-U_k)f_j\right\rangle \right|^2\\
&\ge  \frac{A}{\|V\|^2}\|f\|^2-B\sum_{i\ne k}\|U_i-U_k\|^2\|f\|^2\\
&\ge \left( \frac{A}{\|V\|^2}-(m-1)B\max_{i\ne k}\|U_i-U_k\|^2\right) \|f\|^2.
\end{align*}
This completes the proof.
\end{proof}
 \begin{theorem}\label{thm2}
For $i\in I$, let $F_i=\{f_{ij}\}_{j\in J}$ be a  frame for $\hs$ with bounds $A_i,~B_i$. Assume for any $k\in I$, $\|T_{F_i}-T_{F_k}\|<\frac{A_k}{(m-1)(\sqrt{B_i}+\sqrt{B_k})}$, then the family of frames $\{F_i\}_{i\in I}$ is woven.
 \end{theorem}
 \begin{proof}
Let $\{\sigma_i\}_{i\in I}$ be any partition of $J$. Then, for every $f\in\hs$, we have
$$\sum_{i\in I}\sum_{j\in\sigma_i}\left| \left\langle f,f_{ij}\right\rangle \right| ^2\le \sum_{i\in I}\sum_{j\in J}\left| \left\langle f,f_{ij}\right\rangle \right| ^2\le \sum_{i\in I}B_i\|f\|^2\le m \max_{i\in I}{B_i}\|f\|^2.$$
Hence, the family $\{f_{ij}\}_{j\in\sigma_i,i\in I}$ is a Bessel sequence with Bessel bound $m \max_{i\in I}{B_i}\|f\|^2$.

For any $f\in\hs$, let $T_{F_i}^{\sigma_i}=\sum_{j\in\sigma_i}\left\langle f,f_{ij}\right\rangle f_{ij}$,  then
$$\|T_{F_i}^{\sigma_i}f\|^2=\sum_{j\in\sigma_i}\left| \left\langle f,f_{ij}\right\rangle \right| ^2\le \sum_{j\in J}\left| \left\langle f,f_{ij}\right\rangle \right| ^2=\|T_{F_i}f\|^2\le \|T_{F_i}\|^2\|f\|^2,$$
hence $\|T_{F_i}^{\sigma_i}\|\le  \|T_{F_i}\|$. And then
\begin{align*}
\|T_{F_i}^{\sigma_i}{T_{F_i}^{\sigma_i}}^*-T_{F_k}^{\sigma_i}{T_{F_k}^{\sigma_i}}^*\|&= \|T_{F_i}^{\sigma_i}{T_{F_i}^{\sigma_i}}^*-T_{F_i}^{\sigma_i}{T_{F_k}^{\sigma_i}}^*+T_{F_i}^{\sigma_i}{T_{F_k}^{\sigma_i}}^*-T_{F_k}^{\sigma_i}{T_{F_k}^{\sigma_i}}^*\|\\
&\le \|T_{F_i}^{\sigma_i}({T_{F_i}^{\sigma_i}}^*-{T_{F_k}^{\sigma_i}}^*)\|+\|T_{F_i}^{\sigma_i}({T_{F_i}^{\sigma_i}}^*-{T_{F_k}^{\sigma_i}}^*)\|\\
&\le \|T_{F_i}^{\sigma_i}\|\|T_{F_i}-T_{F_k}\|+\|T_{F_i}-T_{F_k}\|\|T_{F_k}\|\\
&\le (\sqrt{B_i}+\sqrt{B_k})\|T_{F_i}-T_{F_k}\|.
\end{align*}
Therefore, 
\begin{align*}
\sum_{i\in I}S_{F_i}^{\sigma_i}&=S_{F_1}^{\sigma_1}+\cdots+S_{F_i}^{\sigma_i}+\cdots+S_{F_k}^{\sigma_k}+\cdots+S_{F_m}^{\sigma_m}\\
&=S_{F_1}^{\sigma_1}+\cdots+S_{F_i}^{\sigma_i}+\cdots+(S_{F_k}-S_{F_k}^{I\setminus\{k\}})+\cdots+S_{F_m}^{\sigma_m}\\
&=S_{F_k}+S_{F_1}^{\sigma_1}-S_{F_k}^{\sigma_1}+\cdots+S_{F_i}^{\sigma_i}-S_{F_k}^{\sigma_i}+S_{F_m}^{\sigma_m}-S_{F_k}^{\sigma_m}\\
&\ge A_k\cdot I_{\hs}-\sum_{i\in I\setminus\{k\}}\|S_{F_i}-S_{F_k}\|\cdot I_{\hs}\\
&\ge (A_k-\sum_{i\in I\setminus\{k\}}(\sqrt{B_i}+\sqrt{B_k})\|T_{F_i}-T_{F_k}\|)\cdot I_{\hs}.
\end{align*}
Hence,  the sequence $\{f_{ij}\}_{j\in\sigma_i,i\in I}$ is a frame for $\hs$, and the family of frames $\{F_i\}_{i\in I}$ is woven.
 \end{proof}
 \begin{theorem}\label{thm2.7}
	For $i\in I$, let $F_i=\{f_{ij}\}_{j\in J}$ be a  frame for $\hs$ with bounds $A_i,~B_i$. For any $\sigma\subset J$ and a fix $k\in I$, let $P_i^{\sigma}(f)=\sum_{j\in \sigma}\left\langle f,f_{ij}\right\rangle f_{ij}-\sum_{j\in \sigma}\left\langle f,f_{kj}\right\rangle f_{kj}$ for $i\ne k$. If $P_i^{\sigma}$ is a positive linear operator,  then the family of frames $\{F_i\}_{i\in I}$ is woven.
\end{theorem}
\begin{proof}
Let $\{\sigma_i\}_{i\in I}$ be any partition of $J$. Then, for every $f\in\hs$, we have
\begin{align*}
A_k\|f\|^2&\le \sum_{j\in J}\left|\left\langle f,f_{ij}\right\rangle \right|^2\\
&=\sum_{j\in \sigma_1}\left|\left\langle f,f_{kj}\right\rangle \right|^2+\cdots+\sum_{j\in \sigma_i}\left|\left\langle f,f_{kj}\right\rangle \right|^2+\cdots+\sum_{j\in \sigma_m}\left|\left\langle f,f_{kj}\right\rangle \right|^2\\
&=\sum_{j\in \sigma_1}\left|\left\langle f,f_{kj}\right\rangle \right|^2+\cdots+
\left\langle \sum_{j\in \sigma_i}\left\langle f,f_{kj}\right\rangle f_{kj},f\right\rangle 
+\cdots+\sum_{j\in \sigma_m}\left|\left\langle f,f_{kj}\right\rangle \right|^2\\
&\le \sum_{j\in \sigma_1}\left|\left\langle f,f_{kj}\right\rangle \right|^2+\cdots+
\left\langle \sum_{j\in \sigma_i}\left\langle f,f_{ij}\right\rangle f_{ij}-P_i^{\sigma_i}(f),f\right\rangle 
+\cdots+\sum_{j\in \sigma_m}\left|\left\langle f,f_{kj}\right\rangle \right|^2\\
&\le \sum_{j\in \sigma_1}\left|\left\langle f,f_{kj}\right\rangle \right|^2+\cdots+
\left\langle \sum_{j\in \sigma_i}\left\langle f,f_{ij}\right\rangle f_{ij},f\right\rangle 
+\cdots+\sum_{j\in \sigma_m}\left|\left\langle f,f_{kj}\right\rangle \right|^2\\
&=\sum_{j\in \sigma_1}\left|\left\langle f,f_{kj}\right\rangle \right|^2+\cdots+
\sum_{j\in \sigma_i}\left|\left\langle f,f_{ij}\right\rangle \right|^2
+\cdots+\sum_{j\in \sigma_m}\left|\left\langle f,f_{kj}\right\rangle \right|^2\\
&\le \sum_{j\in \sigma_1}\left|\left\langle f,f_{1j}\right\rangle \right|^2+\cdots+
\sum_{j\in \sigma_i}\left|\left\langle f,f_{ij}\right\rangle \right|^2
+\cdots+\sum_{j\in \sigma_m}\left|\left\langle f,f_{mj}\right\rangle \right|^2\\
&\le (B_1+\cdots+B_i+\cdots+B_m)\|f\|^2\\
&=\sum_{i\in I}B_i\|f\|^2,
\end{align*}
thus,
$$A_k\|f\|^2\le \sum_{i\in I}\sum_{j\in \sigma_i}\left|\left\langle f,f_{ij}\right\rangle \right|^2\le \sum_{i\in I}B_i\|f\|^2.$$
The proof is completed.
\end{proof}
\begin{example}
Let $\{e_j\}_{j=1}^{\infty}$ be an orthonormal basis for $\hs$. Let $f_j=e_j+e_{j+1}$ and $g_j=e_j+e_{j+1}+e_{j+2}$, then $F=\{f_j\}_{j=1}^{\infty}$ and $G=\{g_j\}_{j=1}^{\infty}$ are frames for $\hs$ with bounds $1,4$ and $1,9$, respectively.

For any $J\subset \ns$, we compute
$$\sum_{j\in J}\left\langle f,g_i\right\rangle g_i-\left\langle f,f_i\right\rangle f_i=\sum_{j\in J}\left\langle f,e_{i+2}\right\rangle e_{i+2}.$$
For any $J\subset \ns$, define a linear operator $P^J:\hs\rightarrow\hs$ by $P(f)=\sum_{j\in J}\left\langle f,g_j\right\rangle g_j-\left\langle f,f_j\right\rangle f_j$. then for each $f=\sum_{k\in \ns}\left\langle f,e_k\right\rangle e_k$, we have
$$\left\langle P^J(f),f\right\rangle =\left\langle \sum_{j\in J}\left\langle f,e_{j+2}\right\rangle e_{j+2},\sum_{k\in \ns}\left\langle f,e_k\right\rangle e_k\right\rangle =\sum_{j\in J}\left| \left\langle f,e_{j+2}\right\rangle \right| ^2\ge 0.$$
Therefore, $P^J$ is a positive linear operator, Hence by Theorem \ref{thm2.7}, $F$ and $G$ are woven. In fact, for any $\sigma\subset \ns$,
$$\sum_{j\in\sigma}\left| \left\langle f,f_j\right\rangle \right| ^2+\sum_{j\in\sigma^c}\left| \left\langle f,g_j\right\rangle \right| ^2=\sum_{j\in\ns}\left| \left\langle f,d_jf_j+(1-d_j)g_j\right\rangle \right| ^2=\sum_{j\in\ns}\left| \left\langle f,e_j+e_{j+1}\right\rangle \right| ^2+\sum_{j\in\sigma^c}\left| \left\langle f,e_{j+2}\right\rangle \right| ^2,$$
where $d_j^{\sigma}=1$ for $j\in\sigma$ and otherwise $0$.
Thus $\{f_j\}_{j\in\sigma}\cup\{g_j\}_{j\in\sigma^c}$ is a frame for $\hs$ with bounds $1$ and $5$.
\end{example}
 \begin{definition}
	Let $F=\{f_j\}_{j\in J}$ and $G=\{g_j\}_{j\in J}$ be sequences  in Hilbert space $\hs$, let $0< \mu, \lambda< 1$.  If 
	\begin{equation*}
	\sum_{j\in J}\left| \left\langle f,f_j-g_j\right\rangle\right| ^2\le \lambda\sum_{j\in J}\left|  \left\langle f,f_j\right\rangle\right|  ^2+\mu\|f\|^2.
	\end{equation*}
	then we say that $G$ is a $(\lambda,\mu)$-perturbation of $F$.
\end{definition}
 \begin{theorem}\label{thm2.9}
	For $i\in I$, let $F_i=\{f_{ij}\}_{j\in J}$ be a  frame for $\hs$ with bounds $A_i,~B_i$. For a fix $k\in I$, let $F_i$ be the $(\lambda_i,\mu_i)$-perturbation of $F_k$ for all $i\in k$. If 
	$$\sum_{i\ne k}\lambda_i<1{~\rm and ~}A_k>\frac{\sum_{i\ne k}\mu_i}{1-\sum_{i\ne k}\lambda_i},$$
	 then the family of frames $\{F_i\}_{i\in I}$ is woven.
\end{theorem}
\begin{proof}
Let $\{\sigma_i\}_{i\in I}$ be any partition of $J$, 
observe that 
$$\sum_{j\in \sigma_i}\left| \left\langle f,f_{kj}-f_{ij}\right\rangle\right| ^2\le\sum_{j\in J}\left| \left\langle f,f_{kj}-f_{ij}\right\rangle\right| ^2.$$
For every $f\in\hs$, we have
\begin{align*}
\sum_{i\in I}B_i\|f\|^2&\ge \sum_{j\in \sigma_1}\left| \left\langle f,f_{1j}\right\rangle\right| ^2+\cdots+\sum_{j\in \sigma_i}\left| \left\langle f,f_{ij}\right\rangle\right| ^2+\cdots+\sum_{j\in \sigma_k}\left| \left\langle f,f_{kj}\right\rangle\right| ^2+\cdots+\sum_{j\in \sigma_m}\left| \left\langle f,f_{mj}\right\rangle\right| ^2\\
&=\sum_{j\in \sigma_k}\left| \left\langle f,f_{kj}\right\rangle\right| ^2
+\sum_{j\in \sigma_1}\left| \left\langle f,f_{1j}-f_{kj}+f_{kj}\right\rangle\right| ^2+\cdots
+\sum_{j\in \sigma_i}\left| \left\langle f,f_{ij}-f_{kj}+f_{kj}\right\rangle\right| ^2\\
&\qquad+\cdots+\sum_{j\in \sigma_m}\left| \left\langle f,f_{mj}-f_{kj}+f_{kj}\right\rangle\right| ^2\\
&\ge \sum_{j\in \sigma_k}\left| \left\langle f,f_{kj}\right\rangle\right| ^2+\sum_{j\in \sigma_1}\left| \left\langle f,f_{kj}\right\rangle\right| ^2-\sum_{j\in \sigma_1}\left| \left\langle f,f_{kj}-f_{1j}\right\rangle\right| ^2\\
&\qquad+\cdots+\sum_{j\in \sigma_i}\left| \left\langle f,f_{kj}\right\rangle\right| ^2-\sum_{j\in \sigma_i}\left| \left\langle f,f_{kj}-f_{ij}\right\rangle\right| ^2\\
&\qquad+\cdots+\sum_{j\in \sigma_m}\left| \left\langle f,f_{kj}\right\rangle\right| ^2-\sum_{j\in \sigma_m}\left| \left\langle f,f_{kj}-f_{mj}\right\rangle\right| ^2\\
&=\sum_{i\in I}\sum_{j\in \sigma_i}\left| \left\langle f,f_{kj}\right\rangle\right| ^2-\sum_{i\ne k}\sum_{j\in \sigma_i}\left| \left\langle f,f_{kj}-f_{ij}\right\rangle\right| ^2\\
&\ge \sum_{j\in J}\left| \left\langle f,f_{kj}\right\rangle\right| ^2-(m-1)\sum_{j\in J}\left| \left\langle f,f_{kj}-f_{ij}\right\rangle\right| ^2\\
&\ge  \sum_{j\in J}\left| \left\langle f,f_{kj}\right\rangle\right| ^2-(\lambda_1+\cdots+\lambda_i+\cdots+\lambda_m)\sum_{j\in J}\left| \left\langle f,f_{kj}\right\rangle\right| ^2-(\mu_1+\cdots+\mu_i+\cdots+\mu_m)\|f\|^2\\
&=((1-\sum_{i\ne k}\lambda_1)A-\sum_{i\ne k}\mu_i)\|f\|^2.
\end{align*}
Hence the family of frames $\{F_i\}_{i\in I}$ is woven with bounds $(1-\sum_{i\ne k}\lambda_1)A-\sum_{i\ne k}\mu_i$ and $\sum_{i\in I}B_i$.
\end{proof}
\begin{example}
Let $\{e_j\}_{j=1}^{\infty}$ be an orthonormal basis for $\hs$, let $f_j=\frac{4}{3}(e_j+e_{j+1}+e_{j+2})$, $g_j=e_j+e_{j+1}+e_{j+2}$  and $g_j=\frac{5}{6}(e_j+e_{j+1}+e_{j+2})$  for all $j\in\ns$. Then $F=\{f_j\}_{j=1}^{\infty}$, $G=\{g_j\}_{j=1}^{\infty}$ and $H=\{h_j\}_{j=1}^{\infty}$ are frames for $\hs$ with frame bounds $(\frac{16}{9},16)$, $(1,9)$ and $(\frac{25}{36},\frac{25}{4})$,  respectively. 

Choose $\lambda_1=\frac{1}{9}$, $\mu_1=\frac{1}{9}$ and $\lambda_2=\frac{1}{4}$, $\mu_2=\frac{2}{9}$. Then $\lambda_1+\lambda_2<1$ and $A_1>\frac{\mu_1+\mu_2}{1-(\lambda_1+\lambda_2)}$. For any $f\in\hs$, we compute
$$\sum_{j\in \ns}\left| \left\langle f,f_j-g_j\right\rangle \right|^2 =\frac{1}{9}\sum_{j\in \ns}\left| \left\langle f,e_j+e_{j+1}+e_{j+2}\right\rangle \right|^2=\frac{1}{9}\sum_{j\in \ns}\left| \left\langle f,f_j\right\rangle \right|^2\le \lambda_1\sum_{j\in \ns}\left| \left\langle f,f_j\right\rangle \right|^2+\mu_1\|f\|^2,$$
and 
$$\sum_{j\in \ns}\left| \left\langle f,f_j-h_j\right\rangle \right|^2 =\frac{1}{4}\sum_{j\in \ns}\left| \left\langle f,e_j+e_{j+1}+e_{j+2}\right\rangle \right|^2=\frac{1}{4}\sum_{j\in \ns}\left| \left\langle f,f_j\right\rangle \right|^2\le \lambda_2\sum_{j\in \ns}\left| \left\langle f,f_j\right\rangle \right|^2+\mu_2\|f\|^2.$$
Hence by Theorem \ref{thm2.9}, $F$, $G$ and $H$ are woven. In fact, for any partition $\{\sigma_1, \sigma_2,\sigma_3\}$ of $\ns$, 
\begin{align*}
\sum_{j\in \sigma_1}\left| \left\langle f,f_j\right\rangle \right| ^2+\sum_{j\in \sigma_2}\left| \left\langle f,g_j\right\rangle \right| ^2+\sum_{j\in \sigma_3}\left| \left\langle f,h_j\right\rangle \right| ^2&=\sum_{j\in \ns}\left| \left\langle f,d_j^{\sigma_1}f_j+d_j^{\sigma_2}g_j+d_j^{\sigma_3}h_j\right\rangle \right| ^2\\
&=\sum_{j\in \ns}\left| \left\langle f,(\frac{4d_j^{\sigma_1}}{3}+d_j^{\sigma_2}+\frac{5d_j^{\sigma_3}}{6})(e_j+e_{j+1}+e_{j+2})\right\rangle \right| ^2,
\end{align*}

where $d_j^{\sigma_i}=1$ for $j\in\sigma_i$ $(i=1,2,3)$ and otherwise $0$.
Thus $\{f_j\}_{j\in\sigma_1}\cup\{g_j\}_{j\in\sigma_2}\cup\{h_j\}_{j\in\sigma_3}$ is a frame for $\hs$.

\end{example}

The next theorem shows that a family of  invertible operators applied to standard woven frames leaves them woven.
\begin{theorem}
Let the family of frames $\{F_i=\{f_{ij}\}_{j\in J}: i\in I\}$  be woven for $\hs$ with universal bounds $A$ and $B$. Let $T_i\in L(\hs)$ be a invertible operator for $i\in I$, then the family of frames $\{T_iF_i=\{T_if_{ij}\}_{j\in J}: i\in I\}$ is also woven.
	\end{theorem}
\begin{proof}
	Let $\{\sigma_i\}_{i\in I}$ be any partition of $J$, for any $f\in\hs$, we have
\begin{align*}
\sum_{i\in I}\sum_{j\in \sigma_i}\left| \left\langle f,T_if_{ij}\right\rangle \right|^2 &\le \max_{i\in I}\{\|T_i\|^2\}\sum_{i\in I}\sum_{j\in \sigma_i}\left| \left\langle f,f_{ij}\right\rangle \right|^2
\le B\max_{i\in I}\{\|T_i\|^2\}\|f\|^2.
\end{align*}
This gives a universal frame bound. For universal lower frame bound, we compute
\begin{align*}
A\|f\|^2 &\le \sum_{i\in I}\sum_{j\in \sigma_i}\left| \left\langle f,f_{ij}\right\rangle \right|^2\\
&=\sum_{i\in I}\sum_{j\in \sigma_i}\left| \left\langle f,T_i^{-1}T_if_{ij}\right\rangle \right|^2\\
&\le \max{i\in I}\{\|T_i^{-1}\|^2\}\sum_{i\in I}\sum_{j\in \sigma_i}\left| \left\langle f,T_if_{ij}\right\rangle \right|^2.
\end{align*}	
Hence, the family of frames $\{T_iF_i=\{T_if_{ij}\}_{j\in J}: i\in I\}$ is  woven with universal frame bounds
$$A\min_{i\in I}\{\|T_i^{-1}\|^2\}~~~~ {\rm and}~~~~ B\max_{i\in I}\{\|T_i\|^2\}$$
\end{proof}

In the following result, we give some conditions that under those, perturbation of wovens are woven again. 

\begin{definition}
Let $F=\{f_j\}_{j\in J}$ and $G=\{g_j\}_{j\in J}$ be sequences  in Hilbert space $\hs$, let $0<  \lambda< 1$. Let $\{c_j\}_{j\in J}$ be an arbitrary sequence of positive numbers such that $\sum_{j\in J}c_j^2<\infty$. If 
\begin{equation}\label{2.1}
\big\| \sum_{j\in J}c_j(f_j-g_j)\big\|\le \lambda\|\{c_j\}_{j\in J}\|,
\end{equation}
then we say that $G$ is a $\lambda$-perturbation of $F$.
\end{definition}
 
 \begin{theorem}
Let the family of frames $\{F_i=\{f_{ij}\}_{j\in J}\}_{i\in I}$ be woven with bounds $A,B$, and $F_i^{'}=\{f_{ij}^{'}\}_{j\in J}$  be $\lambda_i$-perturbation  of $F_i$. If $\lambda_i<\frac{A}{2\sqrt{mB}}$ for all $i\in I$, then the family of frames $\{F_i^{'}\}_{i\in I}$ is woven in Hilbert space $\hs$.
 \end{theorem}
 \begin{proof}
 Let $T_{F_i}$ be the synthesis operator of $F_i$, then \eqref{2.1} is equal to 
 $$\|T_{F_i}-T_{F_i^{'}}\|\le \lambda_i.$$
Let $\{\sigma_i\}_{i\in I}$ be any partition of $J$ , for every $f\in\hs$
$$A\|f\|^2\le \sum_{i\in I}\|{T_{F_i}^{\sigma_i}}^*(f)\|^2\le B\|f\|^2.$$
Therefore,
\begin{align*}
\sum_{i\in I}\|{T_{F_i}^{\sigma_i}}^*(f)\|^2
&=\|{T_{F_1^{'}}^{\sigma_1}}^*(f)\|^2+\cdots+\|{T_{F_i^{'}}^{\sigma_i}}^*(f)\|^2+\cdots+\|{T_{F_m^{'}}^{\sigma_m}}^*(f)\|^2\\
&\le \|{T_{F_1^{'}}^{\sigma_1}}^*(f)\|^2+\cdots+
(\|{T_{F_i}^{\sigma_i}}^*(f)\|+\|{T_{F_i^{'}}^{\sigma_i}}^*(f)-{T_{F_i}^{\sigma_i}}^*(f)\|)^2
+\cdots+\|{T_{F_m^{'}}^{\sigma_m}}^*(f)\|^2\\
&\le \|{T_{F_1^{'}}^{\sigma_1}}^*(f)\|^2+\cdots+
(\|{T_{F_i}^{\sigma_i}}^*(f)\|+\lambda_i\|f\|)^2
+\cdots+\|{T_{F_m^{'}}^{\sigma_m}}^*(f)\|^2\\
&=\|{T_{F_1^{'}}^{\sigma_1}}^*(f)\|^2+\cdots+
\|{T_{F_i}^{\sigma_i}}^*(f)\|^2+\lambda_i^2\|f\|^2+2\lambda_i\|{T_{F_i}^{\sigma_i}}^*(f)\|\|f\|
+\cdots+\|{T_{F_m^{'}}^{\sigma_m}}^*(f)\|^2\\
&\le \sum_{i\in I}\|{T_{F_i}^{\sigma_i}}^*(f)\|^2+\sum_{i\in I}\lambda_i^2\|f\|^2+2\sum_{i\in I}\lambda_i\|{T_{F_i}^{\sigma_i}}^*(f)\|\|f\|\\
&\le (B+\sum_{i\in I}\lambda_i^2+2\sqrt{B}\max_{i\in I}\{\lambda_i\})\|f\|^2.
\end{align*}
On the other hand, by using the inequality $(a_1+a_2+\cdots+a_m)\le \sqrt{m(a_1^2+a_2^2+\cdots+a_m^2)}$ $(a\ge 0)$, for any $f\in\hs$, we have
\begin{align*}
\sum_{i\in I}\|{T_{F_i}^{\sigma_i}}^*(f)\|^2
&=\|{T_{F_1^{'}}^{\sigma_1}}^*(f)\|^2+\cdots+\|{T_{F_i^{'}}^{\sigma_i}}^*(f)\|^2+\cdots+\|{T_{F_m^{'}}^{\sigma_m}}^*(f)\|^2\\
&=\|{T_{F_1^{'}}^{\sigma_1}}^*(f)\|^2+\cdots+
\|{T_{F_i}^{\sigma_i}}^*(f)+{T_{F_i^{'}}^{\sigma_i}}^*(f)-{T_{F_i}^{\sigma_i}}^*(f)\|^2
+\cdots+\|{T_{F_m^{'}}^{\sigma_m}}^*(f)\|^2\\
&\ge  \|{T_{F_1^{'}}^{\sigma_1}}^*(f)\|^2+\cdots+
(\|{T_{F_i}^{\sigma_i}}^*(f)\|-\|{T_{F_i^{'}}^{\sigma_i}}^*(f)-{T_{F_i}^{\sigma_i}}^*(f)\|)^2
+\cdots+\|{T_{F_m^{'}}^{\sigma_m}}^*(f)\|^2\\
&= \|{T_{F_1^{'}}^{\sigma_1}}^*(f)\|^2+\cdots+\|{T_{F_m^{'}}^{\sigma_m}}^*(f)\|^2\\
&\qquad+\|{T_{F_i}^{\sigma_i}}^*(f)\|^2+\|{T_{F_i^{'}}^{\sigma_i}}^*(f)-{T_{F_i}^{\sigma_i}}^*(f)\|^2-2\|{T_{F_i}^{\sigma_i}}^*(f)\|\|{T_{F_i^{'}}^{\sigma_i}}^*(f)-{T_{F_i}^{\sigma_i}}^*(f)\|\\
&\ge \|{T_{F_1^{'}}^{\sigma_1}}^*(f)\|^2+\cdots+\|{T_{F_m^{'}}^{\sigma_m}}^*(f)\|^2+\|{T_{F_i}^{\sigma_i}}^*(f)\|^2-2\|{T_{F_i}^{\sigma_i}}^*(f)\|\|{T_{F_i^{'}}^{\sigma_i}}^*(f)-{T_{F_i}^{\sigma_i}}^*(f)\|\\
&\ge \sum_{i\in I}\|{T_{F_i}^{\sigma_i}}^*(f)\|^2-2\max_{i\in I}\{\lambda_i\}\|f\|\sum_{i\in I}\|{T_{F_i}^{\sigma_i}}^*(f)\|\\
&\ge A\|f\|^2-2\max_{i\in I}\{\lambda_i\}\|f\|\sqrt{m\sum_{i\in I}\|{T_{F_i}^{\sigma_i}}^*(f)\|^2}\\
&\ge A\|f\|^2-2\sqrt{mB}\max_{i\in I}\{\lambda_i\}\|f\|^2.
\end{align*}
This completes the proof.
 \end{proof}
\bibliographystyle{plain}
\bibliography{my}

\begin{thebibliography}{10}

\bibitem{bemrose2015weaving}
Travis Bemrose, Peter~G Casazza, Karlheinz Gr{\"o}chenig, Mark~C Lammers, and
  Richard~G Lynch.
\newblock Weaving frames.
\newblock {\em Operator and Matrices}, 10(4):1093--1116, 2016.

\bibitem{casazza2016weaving}
Peter~G Casazza, Daniel Freeman, and Richard~G Lynch.
\newblock Weaving schauder frames.
\newblock {\em Journal of Approximation Theory}, 211:42--60, 2016.

\bibitem{casazza2003equal}
Peter~G Casazza and Jelena Kova{\v{c}}evi{\'c}.
\newblock Equal-norm tight frames with erasures.
\newblock {\em Advances in Computational Mathematics}, 18(2-4):387--430, 2003.

\bibitem{casazza2015weaving}
Peter~G Casazza and Richard~G Lynch.
\newblock Weaving properties of hilbert space frames.
\newblock In {\em Sampling Theory and Applications (SampTA), 2015 International
  Conference on}, pages 110--114. IEEE, 2015.

\bibitem{chan2004tight}
Raymond~H Chan, Sherman~D Riemenschneider, Lixin Shen, and Zuowei Shen.
\newblock Tight frame: an efficient way for high-resolution image
  reconstruction.
\newblock {\em Applied and Computational Harmonic Analysis}, 17(1):91--115,
  2004.

\bibitem{daubechies1986painless}
Ingrid Daubechies, Alex Grossmann, and Yves Meyer.
\newblock Painless nonorthogonal expansions.
\newblock {\em Journal of Mathematical Physics}, 27(5):1271--1283, 1986.

\bibitem{duffin1952class}
Richard~J Duffin and Albert~C Schaeffer.
\newblock A class of nonharmonic fourier series.
\newblock {\em Transactions of the American Mathematical Society},
  72(2):341--366, 1952.

\bibitem{eldar2002optimal}
Yonina~C Eldar and G~David Forney.
\newblock Optimal tight frames and quantum measurement.
\newblock {\em IEEE Transactions on Information Theory}, 48(3):599--610, 2002.

\bibitem{han2000frames}
Deguang Han and David~R Larson.
\newblock {\em Frames, bases and group representations}, volume 697.
\newblock American Mathematical Soc., 2000.

\bibitem{han2018recovery}
Deguang Han, Fusheng Lv, and Wenchang Sun.
\newblock Recovery of signals from unordered partial frame coefficients.
\newblock {\em Applied and Computational Harmonic Analysis}, 44(1):38--58,
  2018.

\bibitem{han2014reconstruction}
Deguang Han and Wenchang Sun.
\newblock Reconstruction of signals from frame coefficients with erasures at
  unknown locations.
\newblock {\em IEEE Transactions on Information Theory}, 60(7):4013--4025,
  2014.

\bibitem{li2018frame}
Dongwei Li, Jinsong Leng, Tingzhu Huang, and Qing Gao.
\newblock Frame expansions with probabilistic erasures.
\newblock {\em Digital Signal Processing}, 72:75--82, 2018.

\bibitem{lu2008theory}
Yue~M Lu and Minh~N Do.
\newblock A theory for sampling signals from a union of subspaces.
\newblock {\em IEEE transactions on signal processing}, 56(6):2334--2345, 2008.

\end{thebibliography}

\end{document}